\documentclass{amsart}
\usepackage{amssymb}
\usepackage{amsfonts}

\newtheorem{theorem}{Theorem}
\theoremstyle{plain}

\newtheorem{corollary}{Corollary}

\newtheorem{definition}{Definition}

\newtheorem{lemma}{Lemma}

\numberwithin{equation}{section}
\input{tcilatex}

\begin{document}
\title[Co--ordinated $r-$Convex Functions]{On Hadamard-Type Inequalities for
Co--ordinated $r-$Convex Functions}
\author{$^{\bigstar }$M. Emin \"{O}zdemir}
\address{$^{\bigstar }$Atat\"{u}rk University, K. K. Education Faculty,
Department of Mathematics, 25240, Campus, Erzurum, Turkey }
\email{emos@atauni.edu.tr}
\author{$^{\blacksquare ,\spadesuit }$Ahmet Ocak Akdemir}
\address{$^{\spadesuit }$A\u{g}r\i\ \.{I}brahim \c{C}e\c{c}en University\\
Faculty of Science and Arts, Department of Mathematics, 04100, A\u{g}r\i ,
Turkey}
\email{ahmetakdemir@agri.edu.tr}
\date{August 20, 2010}
\keywords{$r-$Convex, Hadamard-Type Inequalities, Co--ordinates, convex\\
$^{\blacksquare }$ $corresponding$ $author$}

\begin{abstract}
In this paper we defined $r-$convexity on the coordinates and we established
some Hadamard-Type Inequalities.
\end{abstract}

\maketitle

\section{INTRODUCTION}

Let $f:I\subseteq 
\mathbb{R}
\rightarrow 
\mathbb{R}
$ be a convex function defined on the interval $I$ of real numbers and $%
a,b\in I$ with $a<b.$ Then the following inequality holds:%
\begin{equation*}
f\left( \frac{a+b}{2}\right) \leq \frac{1}{b-a}\dint\limits_{a}^{b}f(x)dx%
\leq \frac{f(a)+f(b)}{2}.
\end{equation*}%
This inequality is well known in the literature as Hadamard's inequality.

In \cite{SIMIC}, C.E.M. Pearce, J. Pecaric and V. Simic generalized this
inequality to $r-$convex positive function $f$ which is defined on an
interval $[a,b],$ for all $x,y\in \lbrack a,b]$ and $\lambda \in \lbrack
0,1];$%
\begin{equation*}
f(\lambda x+(1-\lambda )y)\leq \left\{ 
\begin{array}{c}
\left( \lambda \left[ f(x)\right] ^{r}+(1-\lambda )\left[ f(y)\right]
^{r}\right) ^{\frac{1}{r}},if\text{ }r\neq 0 \\ 
\text{ }\left[ f(x)\right] ^{\lambda }\left[ f(y)\right] ^{1-\lambda }\text{
\ \ \ \ \ \ \ \ \ \ \ \ \ \ \ },\text{\ }if\text{ }r=0%
\end{array}%
\right. .
\end{equation*}%
We have that $0-$convex functions are simply log-convex functions and $1-$%
convex functions are ordinary convex functions.

In \cite{PTT}, N.P.G. Ngoc, N.V. Vinh and P.T.T. Hien established following
theorems for $r-$convex functions:

\begin{theorem}
Let $f:[a,b]\rightarrow (0,\infty )$ be $r-$convex function on $[a,b]$ with $%
a<b.$Then the following inequality holds for $0<r\leq 1:$%
\begin{equation}
\frac{1}{b-a}\dint\limits_{a}^{b}f(x)dx\leq \left( \frac{r}{r+1}\right) ^{%
\frac{1}{r}}\left( \left[ f(a)\right] ^{r}+\left[ f(b)\right] ^{r}\right) ^{%
\frac{1}{r}}.  \label{1.1}
\end{equation}
\end{theorem}

\begin{theorem}
Let $f,g:[a,b]\rightarrow (0,\infty )$ be $r-$convex and $s-$convex
functions respectively on $[a,b]$ with $a<b.$ Then the following inequality
holds for $0<r,s\leq 2:$%
\begin{eqnarray}
\frac{1}{b-a}\dint\limits_{a}^{b}f(x)g(x)dx &\leq &\frac{1}{2}\left( \frac{r%
}{r+2}\right) ^{\frac{2}{r}}\left( \left[ f(a)\right] ^{r}+\left[ f(b)\right]
^{r}\right) ^{\frac{2}{r}}  \label{1.2} \\
&&+\frac{1}{2}\left( \frac{s}{s+2}\right) ^{\frac{2}{s}}\left( \left[ g(a)%
\right] ^{s}+\left[ g(b)\right] ^{s}\right) ^{\frac{2}{s}}.  \notag
\end{eqnarray}
\end{theorem}

\begin{theorem}
Let $f,g:[a,b]\rightarrow (0,\infty )$ be $r-$convex and $s-$convex
functions respectively on $[a,b]$ with $a<b.$Then the following inequality
holds if $r>1,$ and $\frac{1}{r}+\frac{1}{s}=1:$%
\begin{equation}
\frac{1}{b-a}\dint\limits_{a}^{b}f(x)g(x)dx\leq \left( \frac{\left[ f(a)%
\right] ^{r}+\left[ f(b)\right] ^{r}}{2}\right) ^{\frac{1}{r}}\left( \frac{%
\left[ g(a)\right] ^{s}+\left[ g(b)\right] ^{s}}{2}\right) ^{\frac{1}{s}}.
\label{1.3}
\end{equation}
\end{theorem}

Similar results can be found for several kind of convexity, in \cite{LATIF}, 
\cite{DAR1}, \cite{DAR2} and \cite{LAT2}.

In \cite{SS}, a convex function on the co-ordinates defined by S.S. Dragomir
as follow:

\begin{definition}
A function $f:\Delta \rightarrow 
\mathbb{R}
$ which is convex on $\Delta $ is called co-ordinated convex on $\Delta $ if
the partial mappings $f_{y}:[a,b]\rightarrow 
\mathbb{R}
,$ $f_{y}(u)=f(u,y)$ and $f_{x}:[c,d]\rightarrow 
\mathbb{R}
,$ $f_{x}(v)=f(x,v)$ are convex for all $y\in \lbrack c,d]$ and $x\in
\lbrack a,b].$
\end{definition}

Again in \cite{SS}, Dragomir gave the following inequalities related to
definition given above.

\begin{theorem}
Suppose that $f:\Delta \rightarrow 
\mathbb{R}
$ is co-ordinated convex on $\Delta .$ Then one has the inequalities:%
\begin{eqnarray}
&&f\left( \frac{a+b}{2},\frac{c+d}{2}\right)  \label{1.4} \\
&\leq &\frac{1}{2}\left[ \frac{1}{b-a}\dint\limits_{a}^{b}f\left( x,\frac{c+d%
}{2}\right) dx+\frac{1}{d-c}\dint\limits_{c}^{d}f\left( \frac{a+b}{2}%
,y\right) dy\right]  \notag \\
&\leq &\frac{1}{(b-a)(d-c)}\dint\limits_{a}^{b}\dint\limits_{c}^{d}f(x,y)dxdy
\notag \\
&\leq &\frac{1}{4}\left[ \frac{1}{b-a}\dint\limits_{a}^{b}f\left( x,c\right)
dx+\frac{1}{b-a}\dint\limits_{a}^{b}f\left( x,d\right) dx\right.  \notag \\
&&\left. +\frac{1}{d-c}\dint\limits_{c}^{d}f\left( a,y\right) dy+\frac{1}{d-c%
}\dint\limits_{c}^{d}f\left( b,y\right) dy\right]  \notag \\
&\leq &\frac{f(a,c)+f(a,d)+f(b,c)+f(b,d)}{4}.  \notag
\end{eqnarray}%
The above inequalities are sharp.
\end{theorem}

In \cite{ALO}, M. Alomari and M. Darus proved some inequalities of the
Hadamard and Jensen types for co-ordinated log-convex functions. In \cite%
{BAK}, M.K. Bakula and J. Pecaric improved several inequalities of Jensen's
type for convex functions on the co-ordinates. In \cite{OZ}, M.E. \"{O}%
zdemir, E. Set and M.Z. Sar\i kaya established Hadamard's type inequalities
for co-ordinated $m-$convex and $(\alpha ,m)-$convex functions. Similar
results can be found in \cite{LATIF}, \cite{DAR1}, \cite{DAR2} and \cite%
{LAT2}.

The main purpose of this present note is to give definition of $r-$convexity
on the coordinates and to prove some Hadamard-type inequalities for
co-ordinated $r-$convex functions.

\section{MAIN\ RESULTS}

We can define $r-$convex functions on the coordinates as follow:

\begin{definition}
A function $f:\Delta =[a,b]\times \lbrack c,d]\rightarrow 
\mathbb{R}
_{+}$ will be called $r-$convex on $\Delta ,$ for all $t,\lambda \in \lbrack
0,1]$ and $(x,y),(u,v)\in \Delta ,$ if the following inequalities hold:%
\begin{eqnarray*}
&&f\left( tx+(1-t)y,\lambda u+(1-\lambda )v\right) \\
&\leq &\left\{ 
\begin{array}{c}
\left[ t\lambda f^{r}(x,u)+t(1-\lambda )f^{r}(x,v)+\left( 1-t\right) \lambda
f^{r}\left( y,u\right) +\left( 1-t\right) \left( 1-\lambda \right)
f^{r}\left( y,v\right) \right] ^{\frac{1}{r}},\ if\text{ }r\neq 0 \\ 
\\ 
f^{t\lambda }(x,u)f^{t(1-\lambda )}(x,v)f^{(1-t)\lambda
}(y,u)f^{(1-t)(1-\lambda )}(y,v),\text{ }if\text{ }r=0%
\end{array}%
\right. .
\end{eqnarray*}
\end{definition}

\bigskip

It is simply to see that if we choose $r=0$, we have co-ordinated log-convex
functions and if we choose $r=1$, we have co-ordinated convex functions. A
function $f:\Delta \rightarrow 
\mathbb{R}
_{+}$ is $r-$convex on $\Delta $ is called co-ordinated $r-$convex on $%
\Delta $ if the partial mappings%
\begin{equation*}
f_{y}:[a,b]\rightarrow 
\mathbb{R}
_{+},\text{ }f_{y}(u)=f(u,y)
\end{equation*}%
and%
\begin{equation*}
f_{x}:[c,d]\rightarrow 
\mathbb{R}
_{+},\text{ }f_{x}(v)=f(x,v)
\end{equation*}%
are $r-$convex for all $y\in \lbrack c,d]$ and $x\in \lbrack a,b].$

We need the foolowing lemma for our main results.

\begin{lemma}
Every $r-$convex mapping $f:\Delta =[a,b]\times \lbrack c,d]\rightarrow 
\mathbb{R}
_{+}$ is $r-$convex on the co-ordinates, where $t,\lambda \in \lbrack 0,1]$.
\end{lemma}

\begin{proof}
Suppose that $f:\Delta =[a,b]\times \lbrack c,d]\rightarrow 
\mathbb{R}
_{+}$ is $r-$convex on $\Delta .$ Consider the mapping%
\begin{equation*}
f_{y}:[a,b]\rightarrow 
\mathbb{R}
_{+},\text{ }f_{y}(u)=f(u,y)
\end{equation*}

Case 1: For $r=0$ and $u_{1},u_{2}\in \lbrack a,b]$, then we have:%
\begin{eqnarray*}
f_{y}(tu_{1}+(1-t)u_{2}) &=&f(tu_{1}+(1-t)u_{2},y) \\
&=&f(tu_{1}+(1-t)u_{2},\lambda y+(1-\lambda )y) \\
&\leq &f^{t\lambda }(u_{1},y)f^{t(1-\lambda )}(u_{1},y)f^{(1-t)\lambda
}(u_{2},y)f^{(1-t)(1-\lambda )}(u_{2},y) \\
&=&f_{y}^{t\lambda }(u_{1})f_{y}^{t(1-\lambda )}(u_{1})f_{y}^{(1-t)\lambda
}(u_{2})f_{y}^{(1-t)(1-\lambda )}(u_{2}).
\end{eqnarray*}

Case 2: For $r\neq 0$ and $u_{1},u_{2}\in \lbrack a,b]$, then we have:%
\begin{eqnarray*}
f_{y}(tu_{1}+(1-t)u_{2}) &=&f(tu_{1}+(1-t)u_{2},y) \\
&=&f(tu_{1}+(1-t)u_{2},\lambda y+(1-\lambda )y) \\
&\leq &\left[ t\lambda f^{r}(u_{1},y)+t(1-\lambda )f^{r}(u_{1},y)\right. \\
&&\left. \left( 1-t\right) \lambda f^{r}\left( u_{2},y\right) +\left(
1-t\right) \left( 1-\lambda \right) f^{r}\left( u_{2},y\right) \right] ^{^{%
\frac{1}{r}}} \\
&=&\left[ t\lambda f_{y}^{r}(u_{1})+t(1-\lambda )f_{y}^{r}(u_{1})\right. \\
&&\left. +\left( 1-t\right) \lambda f_{y}^{r}\left( u_{2}\right) +\left(
1-t\right) \left( 1-\lambda \right) f_{y}^{r}\left( u_{2}\right) \right] ^{%
\frac{1}{r}}.
\end{eqnarray*}%
Therefore $f_{y}(u)=f(u,y)$ is $r-$convex on $[a,b].$ By a similar argument
one can see $f_{x}(v)=f(x,v)$ is $r-$convex on $[c,d].$
\end{proof}

\begin{theorem}
\bigskip Suppose that $f:\Delta \rightarrow 
\mathbb{R}
_{+}$ be a positive co-ordinated $r-$convex function on $\Delta $. If $%
t,\lambda \in \lbrack 0,1]$ and $(x,y),(u,v)\in \Delta ,$ then one has the
inequality:%
\begin{eqnarray}
&&\frac{1}{(b-a)(d-c)}\dint\limits_{a}^{b}\dint\limits_{c}^{d}f(x,y)dxdy
\label{2.1} \\
&\leq &\frac{1}{2}\left( \frac{r}{r+1}\right) ^{\frac{1}{r}}\left[ \frac{1}{%
b-a}\dint\limits_{a}^{b}\left( \left[ f\left( x,c\right) \right] ^{r}+\left[
f\left( x,d\right) \right] ^{r}\right) ^{\frac{1}{r}}dx\right.   \notag \\
&&\left. +\frac{1}{d-c}\dint\limits_{c}^{d}\left( \left[ f\left( a,y\right) %
\right] ^{r}+\left[ f\left( b,y\right) \right] ^{r}\right) ^{\frac{1}{r}}dy%
\right]   \notag
\end{eqnarray}%
where $0<r\leq 1.$
\end{theorem}

\begin{proof}
Since $f:\Delta =[a,b]\times \lbrack c,d]\rightarrow 
\mathbb{R}
_{+}$ is co-ordinated $r-$convex on $\Delta ,$ then the partial mappings%
\begin{equation*}
f_{x}:[c,d]\rightarrow 
\mathbb{R}
_{+},\text{ }f_{x}(v)=f(x,v)
\end{equation*}%
and%
\begin{equation*}
f_{y}:[a,b]\rightarrow 
\mathbb{R}
_{+},\text{ }f_{y}(u)=f(u,y)
\end{equation*}%
are $r-$convex, by inequality (\ref{1.1}), we can write:%
\begin{equation*}
\frac{1}{d-c}\dint\limits_{c}^{d}f_{x}(y)dy\leq \left( \frac{r}{r+1}\right)
^{\frac{1}{r}}\left( \left[ f_{x}(c)\right] ^{r}+\left[ f_{x}(d)\right]
^{r}\right) ^{\frac{1}{r}}
\end{equation*}%
or%
\begin{equation*}
\frac{1}{d-c}\dint\limits_{c}^{d}f(x,y)dy\leq \left( \frac{r}{r+1}\right) ^{%
\frac{1}{r}}\left( \left[ f(x,c)\right] ^{r}+\left[ f(x,d)\right]
^{r}\right) ^{\frac{1}{r}}.
\end{equation*}%
Dividing both side of inequality $(b-a)$ and integrating respect to $x$ on $%
[a,b],$ we get%
\begin{eqnarray}
&&\frac{1}{(b-a)(d-c)}\dint\limits_{a}^{b}\dint\limits_{c}^{d}f(x,y)dxdy
\label{2.2} \\
&\leq &\left( \frac{r}{r+1}\right) ^{\frac{1}{r}}\left[ \frac{1}{\left(
b-a\right) }\dint\limits_{a}^{b}\left[ \left[ f\left( x,c\right) \right]
^{r}+\left[ f\left( x,d\right) \right] ^{r}\right] ^{\frac{1}{r}}dx\right] .
\notag
\end{eqnarray}%
By a similar argument for the mapping, we have%
\begin{equation*}
f_{y}:[a,b]\rightarrow 
\mathbb{R}
_{+},\text{ }f_{y}(u)=f(u,y)
\end{equation*}%
\begin{eqnarray}
&&\frac{1}{(b-a)(d-c)}\dint\limits_{a}^{b}\dint\limits_{c}^{d}f(x,y)dxdy
\label{2.3} \\
&\leq &\left( \frac{r}{r+1}\right) ^{\frac{1}{r}}\left[ \frac{1}{\left(
d-c\right) }\dint\limits_{c}^{d}\left[ \left[ f\left( a,y\right) \right]
^{r}+\left[ f\left( b,y\right) \right] ^{r}\right] ^{\frac{1}{r}}dy\right] .
\notag
\end{eqnarray}%
By addition (\ref{2.2}) and (\ref{2.3}), (\ref{2.1}) is proved.
\end{proof}

\begin{corollary}
In (\ref{2.1}), if we choose $r=1$ we have the mid inequality of (\ref{1.4}).
\end{corollary}

\begin{theorem}
Suppose that $f,g:\Delta \rightarrow 
\mathbb{R}
_{+}$ be co-ordinated $r_{1}-$convex function and co-ordinated $r_{2}-$%
convex function on $\Delta $. Then one has the inequality:%
\begin{eqnarray}
&&\frac{1}{(b-a)(d-c)}\dint\limits_{a}^{b}\dint%
\limits_{c}^{d}f(x,y)g(x,y)dydx  \label{2.4} \\
&\leq &\frac{1}{4}\left( \frac{r_{1}}{r_{1}+2}\right) ^{\frac{2}{r_{1}}}%
\left[ \frac{1}{(b-a)}\dint\limits_{a}^{b}\left( \left[ f(x,c)\right]
^{r_{1}}+\left[ f(x,d)\right] ^{r_{1}}\right) ^{\frac{2}{r_{1}}}dx\right]  
\notag \\
&&+\frac{1}{4}\left( \frac{r_{2}}{r_{2}+2}\right) ^{\frac{2}{r_{2}}}\left[ 
\frac{1}{(b-a)}\dint\limits_{a}^{b}\left( \left[ g(x,c)\right] ^{r_{2}}+%
\left[ g(x,d)\right] ^{r_{2}}\right) ^{\frac{2}{r_{2}}}dx\right]   \notag \\
&&+\frac{1}{4}\left( \frac{r_{1}}{r_{1}+2}\right) ^{\frac{2}{r_{1}}}\left[ 
\frac{1}{(d-c)}\dint\limits_{c}^{d}\left( \left[ f(a,y)\right] ^{r_{1}}+%
\left[ f(b,y)\right] ^{r_{1}}\right) ^{\frac{2}{r_{1}}}dy\right]   \notag \\
&&+\frac{1}{4}\left( \frac{r_{2}}{r_{2}+2}\right) ^{\frac{2}{r_{2}}}\left[ 
\frac{1}{(d-c)}\dint\limits_{c}^{d}\left( \left[ g(a,y)\right] ^{r_{2}}+%
\left[ g(b,y)\right] ^{r_{2}}\right) ^{\frac{2}{r_{2}}}dy\right]   \notag
\end{eqnarray}%
where $r_{1}>0,$ $r_{2}\leq 2.$
\end{theorem}

\begin{proof}
Since $f,g:\Delta =[a,b]\times \lbrack c,d]\rightarrow 
\mathbb{R}
_{+}$ is co-ordinated $r_{1}-$convex and $r_{2}-$convex on $\Delta .$ Then
the partial mappings%
\begin{equation*}
f_{x}:[c,d]\rightarrow 
\mathbb{R}
_{+},\text{ }f_{x}(v)=f(x,v)
\end{equation*}%
and%
\begin{equation*}
f_{y}:[a,b]\rightarrow 
\mathbb{R}
_{+},\text{ }f_{y}(u)=f(u,y)
\end{equation*}%
are $r_{1}-$convex on $\Delta .$ On the other hand the partial mappings%
\begin{equation*}
g_{x}:[c,d]\rightarrow 
\mathbb{R}
_{+},\text{ }g_{x}(v)=g(x,v)
\end{equation*}%
and%
\begin{equation*}
g_{y}:[a,b]\rightarrow 
\mathbb{R}
_{+},\text{ }g_{y}(u)=g(u,y)
\end{equation*}%
are $r_{2}-$convex on $\Delta .$ From (\ref{1.2}), we get%
\begin{eqnarray*}
\frac{1}{d-c}\dint\limits_{c}^{d}f_{x}(y)g_{x}(y)dy &\leq &\frac{1}{2}\left( 
\frac{r_{1}}{r_{1}+2}\right) ^{\frac{2}{r_{1}}}\left( \left[ f_{x}(c)\right]
^{r_{1}}+\left[ f_{x}(d)\right] ^{r_{1}}\right) ^{\frac{2}{r_{1}}} \\
&&+\frac{1}{2}\left( \frac{r_{2}}{r_{2}+2}\right) ^{\frac{2}{r_{2}}}\left( %
\left[ g_{x}(c)\right] ^{r_{2}}+\left[ g_{x}(d)\right] ^{r_{2}}\right) ^{%
\frac{2}{r_{2}}}
\end{eqnarray*}%
or%
\begin{eqnarray*}
\frac{1}{d-c}\dint\limits_{c}^{d}f(x,y)g(x,y)dy &\leq &\frac{1}{2}\left( 
\frac{r_{1}}{r_{1}+2}\right) ^{\frac{2}{r_{1}}}\left( \left[ f(x,c)\right]
^{r_{1}}+\left[ f(x,d)\right] ^{r_{1}}\right) ^{\frac{2}{r_{1}}} \\
&&+\frac{1}{2}\left( \frac{r_{2}}{r_{2}+2}\right) ^{\frac{2}{r_{2}}}\left( %
\left[ g(x,c)\right] ^{r_{2}}+\left[ g(x,d)\right] ^{r_{2}}\right) ^{\frac{2%
}{r_{2}}}.
\end{eqnarray*}%
Dividing both side of inequality $(b-a)$ and integrating respect to $x$ on $%
[a,b],$ we have%
\begin{eqnarray}
&&\frac{1}{(b-a)(d-c)}\dint\limits_{a}^{b}\dint%
\limits_{c}^{d}f(x,y)g(x,y)dydx  \label{2.5} \\
&\leq &\frac{1}{2}\left( \frac{r_{1}}{r_{1}+2}\right) ^{\frac{2}{r_{1}}}%
\left[ \frac{1}{(b-a)}\dint\limits_{a}^{b}\left( \left[ f(x,c)\right]
^{r_{1}}+\left[ f(x,d)\right] ^{r_{1}}\right) ^{\frac{2}{r_{1}}}dx\right]  
\notag \\
&&+\frac{1}{2}\left( \frac{r_{2}}{r_{2}+2}\right) ^{\frac{2}{r_{2}}}\left[ 
\frac{1}{(b-a)}\dint\limits_{a}^{b}\left( \left[ g(x,c)\right] ^{r_{2}}+%
\left[ g(x,d)\right] ^{r_{2}}\right) ^{\frac{2}{r_{2}}}dx\right] .  \notag
\end{eqnarray}%
By a similar argument, we have%
\begin{eqnarray}
&&\frac{1}{(b-a)(d-c)}\dint\limits_{a}^{b}\dint%
\limits_{c}^{d}f(x,y)g(x,y)dydx  \label{2.6} \\
&\leq &\frac{1}{2}\left( \frac{r_{1}}{r_{1}+2}\right) ^{\frac{2}{r_{1}}}%
\left[ \frac{1}{(d-c)}\dint\limits_{c}^{d}\left( \left[ f(a,y)\right]
^{r_{1}}+\left[ f(b,y)\right] ^{r_{1}}\right) ^{\frac{2}{r_{1}}}dy\right]  
\notag \\
&&+\frac{1}{2}\left( \frac{r_{2}}{r_{2}+2}\right) ^{\frac{2}{r_{2}}}\left[ 
\frac{1}{(d-c)}\dint\limits_{c}^{d}\left( \left[ g(a,y)\right] ^{r_{2}}+%
\left[ g(b,y)\right] ^{r_{2}}\right) ^{\frac{2}{r_{2}}}dy\right] .  \notag
\end{eqnarray}%
Addition (\ref{2.5}) and (\ref{2.6}), we can write%
\begin{eqnarray*}
&&\frac{1}{(b-a)(d-c)}\dint\limits_{a}^{b}\dint%
\limits_{c}^{d}f(x,y)g(x,y)dydx \\
&\leq &\frac{1}{4}\left( \frac{r_{1}}{r_{1}+2}\right) ^{\frac{2}{r_{1}}}%
\left[ \frac{1}{(b-a)}\dint\limits_{a}^{b}\left( \left[ f(x,c)\right]
^{r_{1}}+\left[ f(x,d)\right] ^{r_{1}}\right) ^{\frac{2}{r_{1}}}dx\right]  \\
&&+\frac{1}{4}\left( \frac{r_{2}}{r_{2}+2}\right) ^{\frac{2}{r_{2}}}\left[ 
\frac{1}{(b-a)}\dint\limits_{a}^{b}\left( \left[ g(x,c)\right] ^{r_{2}}+%
\left[ g(x,d)\right] ^{r_{2}}\right) ^{\frac{2}{r_{2}}}dx\right]  \\
&&+\frac{1}{4}\left( \frac{r_{1}}{r_{1}+2}\right) ^{\frac{2}{r_{1}}}\left[ 
\frac{1}{(d-c)}\dint\limits_{c}^{d}\left( \left[ f(a,y)\right] ^{r_{1}}+%
\left[ f(b,y)\right] ^{r_{1}}\right) ^{\frac{2}{r_{1}}}dy\right]  \\
&&+\frac{1}{4}\left( \frac{r_{2}}{r_{2}+2}\right) ^{\frac{2}{r_{2}}}\left[ 
\frac{1}{(d-c)}\dint\limits_{c}^{d}\left( \left[ g(a,y)\right] ^{r_{2}}+%
\left[ g(b,y)\right] ^{r_{2}}\right) ^{\frac{2}{r_{2}}}dy\right] 
\end{eqnarray*}%
which completes the proof.
\end{proof}

\begin{corollary}
In (\ref{2.4}), if we choose $r_{1}=r_{2}=2,$ we have%
\begin{eqnarray*}
&&\frac{1}{(b-a)(d-c)}\dint\limits_{a}^{b}\dint%
\limits_{c}^{d}f(x,y)g(x,y)dydx \\
&\leq &\frac{1}{8(b-a)}\left( \dint\limits_{a}^{b}\left[ f(x,c)\right]
^{2}dx+\dint\limits_{a}^{b}\left[ f(x,d)\right] ^{2}dx+\dint\limits_{a}^{b}%
\left[ g(x,c)\right] ^{2}dx+\dint\limits_{a}^{b}\left[ g(x,d)\right]
^{2}dx\right) \\
&&+\frac{1}{8(d-c)}\left( \dint\limits_{c}^{d}\left[ f(a,y)\right]
^{2}dy+\dint\limits_{c}^{d}\left[ f(b,y)\right] ^{2}dy+\dint\limits_{c}^{d}%
\left[ g(a,y)\right] ^{2}dy+\dint\limits_{c}^{d}\left[ g(b,y)\right]
^{2}dy\right) .
\end{eqnarray*}
\end{corollary}

\begin{corollary}
In (\ref{2.4}), if we choose $r_{1}=r_{2}=2,$ and $f(x,y)=g(x,y),$ we have%
\begin{eqnarray*}
&&\frac{1}{(b-a)(d-c)}\dint\limits_{a}^{b}\dint\limits_{c}^{d}f(x,y)^{2}dydx
\\
&\leq &\frac{1}{4(b-a)}\left( \dint\limits_{a}^{b}\left[ f(x,c)\right]
^{2}dx+\dint\limits_{a}^{b}\left[ f(x,d)\right] ^{2}dx\right) +\frac{1}{%
4(d-c)}\left( \dint\limits_{c}^{d}\left[ f(a,y)\right] ^{2}dy+\dint%
\limits_{c}^{d}\left[ f(b,y)\right] ^{2}dy\right) .
\end{eqnarray*}
\end{corollary}

\begin{theorem}
Suppose that $f,g:\Delta \rightarrow 
\mathbb{R}
_{+}$ be co-ordinated $r_{1}-$convex function and co-ordinated $r_{2}-$%
convex function on $\Delta $. Then one has the inequality:%
\begin{eqnarray}
&&\frac{1}{(b-a)\left( d-c\right) }\dint\limits_{a}^{b}\dint%
\limits_{c}^{d}f(x,y)g(x,y)dydx  \label{2.7} \\
&\leq &\frac{1}{2}\left( \frac{1}{(b-a)}\dint\limits_{a}^{b}\left( \frac{%
\left[ f(x,c)\right] ^{r_{1}}+\left[ f(x,d)\right] ^{r_{1}}}{2}\right) ^{%
\frac{1}{r_{1}}}dx\right)   \notag \\
&&\times \left( \frac{1}{(b-a)}\dint\limits_{a}^{b}\left( \frac{\left[ g(x,c)%
\right] ^{r_{2}}+\left[ g(x,d)\right] ^{r_{2}}}{2}\right) ^{\frac{1}{r_{2}}%
}dx\right)   \notag \\
&&+\frac{1}{2}\left( \frac{1}{(d-c)}\dint\limits_{c}^{d}\left( \frac{\left[
f(a,y)\right] ^{r_{1}}+\left[ f(b,y)\right] ^{r_{1}}}{2}\right) ^{\frac{1}{%
r_{1}}}dy\right)   \notag \\
&&\times \left( \frac{1}{(d-c)}\dint\limits_{c}^{d}\left( \frac{\left[ g(a,y)%
\right] ^{r_{2}}+\left[ g(b,y)\right] ^{r_{2}}}{2}\right) ^{\frac{1}{r_{2}}%
}dy\right)   \notag
\end{eqnarray}%
where $r_{1}>1$ and $\frac{1}{r_{1}}+\frac{1}{r_{2}}=1.$
\end{theorem}

\begin{proof}
Since $f,g:\Delta =[a,b]\times \lbrack c,d]\rightarrow 
\mathbb{R}
_{+}$ is co-ordinated $r_{1}-$convex and $r_{2}-$convex on $\Delta .$ Then
the partial mappings%
\begin{equation*}
f_{x}:[c,d]\rightarrow 
\mathbb{R}
_{+},\text{ }f_{x}(v)=f(x,v)
\end{equation*}%
and%
\begin{equation*}
f_{y}:[a,b]\rightarrow 
\mathbb{R}
_{+},\text{ }f_{y}(u)=f(u,y)
\end{equation*}%
are $r_{1}-$convex on $\Delta .$ On the other hand the partial mappings%
\begin{equation*}
g_{x}:[c,d]\rightarrow 
\mathbb{R}
_{+},\text{ }g_{x}(v)=g(x,v)
\end{equation*}%
and%
\begin{equation*}
g_{y}:[a,b]\rightarrow 
\mathbb{R}
_{+},\text{ }g_{y}(u)=g(u,y)
\end{equation*}%
are $r_{2}-$convex on $\Delta .$ From (\ref{1.3}), we can write%
\begin{eqnarray*}
&&\frac{1}{d-c}\dint\limits_{c}^{d}f(x,y)g(x,y)dy \\
&\leq &\left( \frac{\left[ f(x,a)\right] ^{r_{1}}+\left[ f(x,b)\right]
^{r_{1}}}{2}\right) ^{\frac{1}{r_{1}}}\left( \frac{\left[ g(x,a)\right]
^{r_{2}}+\left[ g(x,b)\right] ^{r_{2}}}{2}\right) ^{\frac{1}{r_{2}}}.
\end{eqnarray*}%
Integrating this inequality respect to $x$ on $[a,b],$ we get%
\begin{eqnarray}
&&\frac{1}{(b-a)\left( d-c\right) }\dint\limits_{a}^{b}\dint%
\limits_{c}^{d}f(x,y)g(x,y)dydx  \label{2.8} \\
&\leq &\left( \frac{1}{(b-a)}\dint\limits_{a}^{b}\left( \frac{\left[ f(x,c)%
\right] ^{r_{1}}+\left[ f(x,d)\right] ^{r_{1}}}{2}\right) ^{\frac{1}{r_{1}}%
}dx\right)   \notag \\
&&\times \left( \frac{1}{(b-a)}\dint\limits_{a}^{b}\left( \frac{\left[ g(x,c)%
\right] ^{r_{2}}+\left[ g(x,d)\right] ^{r_{2}}}{2}\right) ^{\frac{1}{r_{2}}%
}dx\right) .  \notag
\end{eqnarray}%
Similarly, we can write%
\begin{eqnarray}
&&\frac{1}{(b-a)\left( d-c\right) }\dint\limits_{a}^{b}\dint%
\limits_{c}^{d}f(x,y)g(x,y)dydx  \label{2.9} \\
&\leq &\left( \frac{1}{(d-c)}\dint\limits_{c}^{d}\left( \frac{\left[ f(a,y)%
\right] ^{r_{1}}+\left[ f(b,y)\right] ^{r_{1}}}{2}\right) ^{\frac{1}{r_{1}}%
}dy\right)   \notag \\
&&\times \left( \frac{1}{(d-c)}\dint\limits_{c}^{d}\left( \frac{\left[ g(a,y)%
\right] ^{r_{2}}+\left[ g(b,y)\right] ^{r_{2}}}{2}\right) ^{\frac{1}{r_{2}}%
}dy\right) .  \notag
\end{eqnarray}%
Adding (\ref{2.8}) and (\ref{2.9}), (\ref{2.7}) is proved.
\end{proof}

\begin{corollary}
In (\ref{2.7}), if we choose $r_{1}=r_{2}=2,$ we have%
\begin{eqnarray}
&&\frac{1}{(b-a)\left( d-c\right) }\dint\limits_{a}^{b}\dint%
\limits_{c}^{d}f(x,y)g(x,y)dydx  \label{2.10} \\
&\leq &\frac{1}{2}\sqrt{\frac{1}{2(b-a)}\dint\limits_{a}^{b}\left[ f(x,c)%
\right] ^{2}dx+\frac{1}{2(b-a)}\dint\limits_{a}^{b}\left[ f(x,d)\right]
^{2}dx}  \notag \\
&&\times \sqrt{\frac{1}{2(b-a)}\dint\limits_{a}^{b}\left[ g(x,c)\right]
^{2}dx+\frac{1}{2(b-a)}\dint\limits_{a}^{b}\left[ g(x,d)\right] ^{2}dx} 
\notag \\
&&+\frac{1}{2}\sqrt{\frac{1}{2(d-c)}\dint\limits_{c}^{d}\left[ f(a,y)\right]
^{2}dy+\frac{1}{2(d-c)}\dint\limits_{c}^{d}\left[ f(b,y)\right] dy}  \notag
\\
&&\times \sqrt{\frac{1}{2(d-c)}\dint\limits_{c}^{d}\left[ g(a,y)\right]
^{2}dy+\frac{1}{2(d-c)}\dint\limits_{c}^{d}\left[ g(b,y)\right] ^{2}dy}. 
\notag
\end{eqnarray}
\end{corollary}

\begin{corollary}
In (\ref{2.7}), if we choose $r_{1}=r_{2}=2,$ and $f(x,y)=g(x,y),$we have%
\begin{eqnarray}
&&\frac{1}{(b-a)\left( d-c\right) }\dint\limits_{a}^{b}\dint%
\limits_{c}^{d}f(x,y)^{2}dydx  \label{2.11} \\
&\leq &\frac{1}{4(b-a)}\left[ \dint\limits_{a}^{b}\left[ f(x,c)\right]
^{2}dx+\dint\limits_{a}^{b}\left[ f(x,d)\right] ^{2}dx\right]  \notag \\
&&+\frac{1}{4(d-c)}\left[ \dint\limits_{c}^{d}\left[ f(a,y)\right]
^{2}dy+\dint\limits_{c}^{d}\left[ f(b,y)\right] ^{2}dy\right] .  \notag
\end{eqnarray}
\end{corollary}

\end{document}